\documentclass[twoside]{article}
\usepackage{amssymb}
\usepackage{amsthm}
\usepackage{amsmath}
\usepackage{enumerate}
\usepackage{url}
\usepackage{graphicx}
\usepackage{hyperref}
\usepackage{subcaption}
\usepackage[title]{appendix}
\usepackage[american]{babel}

\newcommand{\textlineskip}{\baselineskip=11.5dd plus0.5dd minus0.2dd}
\makeatletter
\renewenvironment{thebibliography}[1]
{
  \par \bigskip \setcounter{subsect}{0}
  \setcounter{equation}{0}
  \addtocounter{section}{1}\begin{center}\thesection.
    \uppercase{References} \end{center} \par \smallskip
  \@mkboth{\MakeUppercase\refname}{\MakeUppercase\refname}%
  \list{\@biblabel{\@arabic\c@enumiv}}%
  {\settowidth\labelwidth{\@biblabel{#1}}%
    \leftmargin\labelwidth
    \advance\leftmargin\labelsep
    \@openbib@code
    \usecounter{enumiv}%
    \let\p@enumiv\@empty
    \renewcommand\theenumiv{\@arabic\c@enumiv}}%
  \sloppy
  \clubpenalty4000
  \@clubpenalty \clubpenalty
  \widowpenalty4000%
  \sfcode`\.\@m}
{\def\@noitemerr
  {\@latex@warning{Empty `thebibliography' environment}}%
  \endlist}
\makeatother

\def\qed{\hfill $\square$\medskip}

\newtheorem{theorem}{\noindent\hspace*{5mm}Theorem}
\newtheorem*{maintheorem}{\noindent\hspace*{5mm}Main Theorem}

\newtheorem{proposition}[theorem]{\noindent\hspace*{5mm}Proposition}

\newtheorem{conjecture}[theorem]{\noindent\hspace*{5mm}Conjecture}

\newtheorem{definition}[theorem]{\noindent\hspace*{5mm}Definition}

\newtheorem{algorithm}[theorem]{\noindent\hspace*{5mm}Algorithm}

\newtheorem{remark}[theorem]{\noindent\hspace*{5mm}Remark}

\numberwithin{equation}{section}
\numberwithin{theorem}{section}

\catcode`\@=11

\font\twrm=cmr12  
  
  \font\eightrm=cmr8

\newcounter{subsect}

\def\section#1{\par \bigskip \setcounter{subsect}{0}
\setcounter{equation}{0}

\setcounter{theorem}{0}
\addtocounter{section}{1}
\begin{center}\thesection. \uppercase{#1}
\end{center} \par \smallskip }
\def\subsection#1{\par \medskip
\addtocounter{subsect}{1}\begin{center}{ \eightrm
\thesection.\thesubsect. \uppercase{#1}} \end{center} \par
\smallskip }

\newcommand{\abs}[1]{|#1|}
\newcommand{\set}[1]{\left\{#1\right\}}

\newcommand{\arrowsv}[0]{\overset{v}{\rightarrow}}

\newcommand{\mH}[0]{\mathcal{H}}

\DeclareMathOperator{\V}{V}
\DeclareMathOperator{\E}{E}

\begin{document}

\begin{center} {\twrm ANNUAL OF SOFIA UNIVERSITY ``ST.~KLIMENT~OHRIDSKI" \\[3pt]
FACULTY OF MATHEMATICS AND INFORMATICS} \normalsize\textlineskip
\thispagestyle{empty} \setcounter{page}{1}

\vspace*{1.0in}

{\twrm \uppercase{LOWER BOUNDING THE FOLKMAN NUMBERS}\\ \large{$F_v(a_1, ..., a_s; m - 1)$}\\

\vspace*{2cc} {\eightrm ALEKSANDAR BIKOV, NEDYALKO NENOV}}
\end{center}

\vspace*{30dd}

{\parindent0pt \footnotesize \leftskip20pt \rightskip20pt
\baselineskip10pt

For a graph $G$ the expression $G \overset{v}{\rightarrow} (a_1, ..., a_s)$ means that for every $s$-coloring of the vertices of $G$ there exists $i \in \{1, ..., s\}$ such that there is a monochromatic $a_i$-clique of color $i$. The vertex Folkman numbers
$$F_v(a_1, ..., a_s; m - 1) = \min\{\vert V(G)\vert : G \overset{v}{\rightarrow} (a_1, ..., a_s) \mbox{ and } K_{m - 1} \not\subseteq G\}.$$
are considered, where $m = \sum_{i = 1}^{s}(a_i - 1) + 1$. We know the exact values of all the numbers $F_v(a_1, ..., a_s; m - 1)$ when $\max\{a_1, ..., a_s\} \leq 6$ and also the number $F_v(2, 2, 7; 8) = 20$. In \cite{BN15a} we present a method for obtaining lower bounds on these numbers. With the help of this method and a new improved algorithm, in the special case when $\max\{a_1, ..., a_s\} = 7$ we prove that $F_v(a_1, ..., a_s; m - 1) \geq m + 11$ and this bound is exact for all $m$. The known upper bound for these numbers is $m + 12$. At the end of the paper we also prove the lower bounds $19 \leq F_v(2, 2, 2, 4; 5)$ and $29 \leq F_v(7, 7; 8)$.

\smallskip

{\bf  2000 Math.\ Subject Classification:}  05C35
\par
}

\vspace*{16dd}

\section{Introduction}

Only finite, non-oriented graphs without loops and multiple edges are considered in this paper. $G_1 + G_2$ denotes the graph $G$ for which $\V(G) = \V(G_1) \cup \V(G_2)$ and $\E(G) = \E(G_1) \cup \E(G_2) \cup E'$, where $E' = \set{[x, y] : x \in \V(G_1), y \in \V(G_2)}$, i.e. $G$ is obtained by connecting with an edge every vertex of $G_1$ to every vertex of $G_2$. All undefined terms can be found in \cite{W01}. 

Let $a_1, ..., a_s$ be positive integers. The expression $G \arrowsv (a_1, ..., a_s)$ means that for every coloring of $\V(G)$ in $s$ colors ($s$-coloring) there exists $i \in \set{1, ..., s}$ such that there is a monochromatic $a_i$-clique of color $i$. In particular, $G \arrowsv (a_1)$ means that $\omega(G) \geq a_1$. Further, for convenience, instead of $G \arrowsv (\underbrace{2, ..., 2}_r)$ we write $G \arrowsv (2_r)$ and instead of $G \arrowsv (\underbrace{2, ..., 2}_r, a_1, ..., a_s)$ we write $G \arrowsv (2_r, a_1, ..., a_s)$.

Define:

$\mH(a_1, ..., a_s; q) = \set{ G : G \arrowsv (a_1, ..., a_s) \mbox{ and } \omega(G) < q }.$

$\mH(a_1, ..., a_s; q; n) = \set{ G : G \in \mH(a_1, ..., a_s; q) \mbox{ and } \abs{\V(G)} = n }.$

The vertex Folkman number $F_v(a_1, ..., a_s; q)$ is defined by the equality:
\begin{equation*}
F_v(a_1, ..., a_s; q) = \min\set{\abs{\V(G)} : G \in \mH(a_1, ..., a_s; q)}.
\end{equation*}

The graph $G$ is called an extremal graph in $\mH(a_1, ..., a_s; q)$ if $G \in \mH(a_1, ..., a_s;q )$ and $\abs{\V(G)} = F_v(a_1, ..., a_s; q)$. We denote by $\mH_{extr}(a_1, ..., a_s; q)$ the set of all extremal graphs in $\mH(a_1, ..., a_s; q)$.

Folkman proves in \cite{Fol70} that:
\begin{equation}
\label{equation: F_v(a_1, ..., a_s; q) exists}
F_v(a_1, ..., a_s; q) \mbox{ exists } \Leftrightarrow q > \max\set{a_1, ..., a_s}.
\end{equation}
Other proofs of (\ref{equation: F_v(a_1, ..., a_s; q) exists}) are given in \cite{DR08} and \cite{LRU01}. In the special case $s = 2$, a very simple proof of this result is given in \cite{Nen85} with the help of corona product of graphs.\\
Obviously $F_v(a_1, ..., a_s; q)$ is a symmetric function of $a_1, ..., a_s$, and if $a_i = 1$, then
\begin{equation*}
F_v(a_1, ..., a_s; q) = F_v(a_1, ..., a_{i-1}, a_{i+1}, ..., a_s; q).
\end{equation*}
Therefore, it is enough to consider only such Folkman numbers $F_v(a_1, ..., a_s; q)$ for which
\begin{equation}
\label{equation: 2 leq a_1 leq ... leq a_s}
2 \leq a_1 \leq ... \leq a_s.
\end{equation}
We call the numbers $F_v(a_1, ..., a_s; q)$ for which the inequalities (\ref{equation: 2 leq a_1 leq ... leq a_s}) hold canonical vertex Folkman numbers.\\
In \cite{LU96} for arbitrary positive integers $a_1, ..., a_s$ the following terms are defined
\begin{equation}
\label{equation: m and p}
m(a_1, ..., a_s) = m = \sum\limits_{i=1}^s (a_i - 1) + 1 \quad \mbox{ and } \quad p = \max\set{a_1, ..., a_s}.
\end{equation}
It is easy to see that $K_m \arrowsv (a_1, ..., a_s)$ and $K_{m - 1} \not\arrowsv (a_1, ..., a_s)$. Therefore
\begin{equation*}
F_v(a_1, ..., a_s; q) = m, \quad q \geq m + 1.
\end{equation*}
The following theorem for the numbers $F_v(a_1, ..., a_s; m)$ is true:
\begin{theorem}
\label{theorem: F_v(a_1, ..., a_s; m) = m + p}
Let $a_1, ..., a_s$ be positive integers and let $m$ and $p$ be defined by the equalities (\ref{equation: m and p}). If $m \geq p + 1$, then:

(a) $F_v(a_1, ..., a_s; m) = m + p$, \cite{LU96},\cite{LRU01}.

(b) $K_{m+p} - C_{2p + 1} = K_{m - p - 1} + \overline{C}_{2p + 1}$\\ is the only extremal graph in $\mH(a_1, ..., a_s; m)$, \ \cite{LRU01}.
\end{theorem}
The condition $m \geq p + 1$ is necessary according to (\ref{equation: F_v(a_1, ..., a_s; q) exists}). Other proofs of Theorem \ref{theorem: F_v(a_1, ..., a_s; m) = m + p} are given in \cite{Nen00} and \cite{Nen01}.\\

Very little is known about the numbers $F_v(a_1, ..., a_s; m - 1)$. According to (\ref{equation: F_v(a_1, ..., a_s; q) exists}) we have
\begin{equation}
\label{equation: F_v(a_1, ..., a_s; m - 1) exists}
F_v(a_1, ..., a_s; m - 1) \mbox{ exists } \Leftrightarrow m \geq p + 2.
\end{equation}

The following general bounds are known:
\begin{equation}
\label{equation: m + p + 2 leq F_v(a_1, ..., a_s; m - 1) leq m + 3p}
m + p + 2 \leq F_v(a_1, ..., a_s; m - 1) \leq m + 3p,
\end{equation}
where the lower bound is true if $p \geq 2$ and the upper bound is true if $p \geq 3$. The lower bound is obtained in \cite{Nen00} and the upper bound is obtained in \cite{KN06a}. In the border case $m = p + 2$ the upper bounds in (\ref{equation: m + p + 2 leq F_v(a_1, ..., a_s; m - 1) leq m + 3p}) are significantly improved in \cite{SXP09}.

We know all the numbers $F_v(a_1, ..., a_s; m - 1)$ when $\max\set{a_1, ..., a_s} \leq 6$, see \cite{BN17} for details. Regarding the numbers $F_v(a_1, ..., a_s; m - 1)$ when $\max\set{a_1, ..., a_s} = 7$ it is known that $F_v(2, 2, 7; 8) = 20$ \cite{BN17}, and 
$$m + 10 \leq F_v(a_1, ..., a_s; m - 1) \leq m + 12, \cite{BN17}.$$

The lower bound $F_v(2, 2, 7; 8) \geq 20$ is obtained with the help of Algorithm \ref{algorithm: mH_(max)(a_1, ..., a_s; q; n)}, and the upper bound is obtained by constructing 20-vertex graphs in $\mH(2, 2, 7; 8)$. An example for such a graph is given on Figure \ref{figure: H(2, 2, 7; 8; 20)}.

In this paper we present an algorithm (Algorithm \ref{algorithm: mH_(max)(a_1, ..., a_s; q; n), cone vertices}), with the help of which we can obtain lower bounds on the numbers $F_v(a_1, ..., a_s; m - 1)$. Using Algorithm \ref{algorithm: mH_(max)(a_1, ..., a_s; q; n), cone vertices} and $F_v(2, 2, 7; 8) = 20$ we improve the lower bound on the numbers $F_v(a_1, ..., a_s; m - 1)$ when $\max\set{a_1, ..., a_s} = 7$ by proving the following:
\begin{maintheorem}
\label{maintheorem: F_v(a_1, ..., a_s; m - 1) geq m + 11}
Let $a_1, ..., a_s$ be positive integers, such that $\max\set{a_1, ..., a_s} = 7$ and $m = \sum\limits_{i=1}^s (a_i - 1) + 1 \geq 9$. Then
$$F_v(a_1, ..., a_s; m - 1) \geq m + 11.$$
\end{maintheorem}

\begin{remark}
\label{remark: m geq 9}
According to (\ref{equation: F_v(a_1, ..., a_s; m - 1) exists}) the condition $m \geq 9$ in the Main Theorem is necessary.
\end{remark}

\begin{figure}[h]
	\centering
	\begin{subfigure}{\textwidth}
		\centering
		\includegraphics[trim={0 480 0 0},clip,height=220px,width=220px]{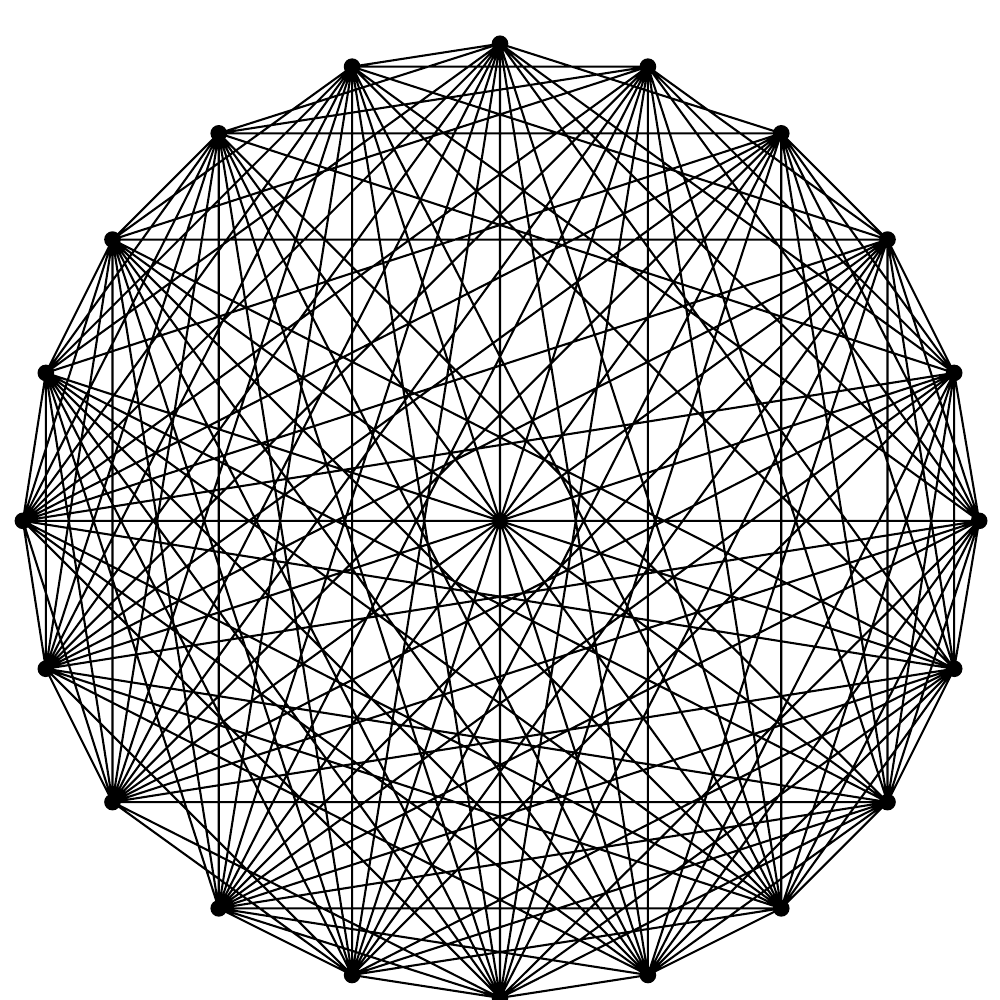}
	\end{subfigure}%
	\caption{An example of a 20-vertex graph in $\mH(2, 2, 7; 8)$ from \cite{BN17}}
	\label{figure: H(2, 2, 7; 8; 20)}
\end{figure}

\section{Bounds on the numbers \lowercase{\large{$\uppercase{F}_v(a_1, ..., a_s; q)$}}}

Let $m$ and $p$ be positive integers. Denote by $\mathcal{S}(m, p)$ the set of all $(b_1, ..., b_r)$ (r is not fixed), where $b_i$ are positive integers such that $\max\set{b_1, ..., b_r} = p$ and $\sum_{i = 1}^{r}(b_i - 1) + 1 = m$. Let $(a_1, ..., a_s) \in \mathcal{S}(m, p)$. Then obviously
$$\min_{(b_1, ..., b_r) \in \mathcal{S}(m, p)} F_v(b_1, ..., b_r; q) \leq F_v(a_1, ..., a_s; q) \leq \max_{(b_1, ..., b_r) \in \mathcal{S}(m, p)} F_v(b_1, ..., b_r; q).$$
Note that $(2_{m - p}, p) \in \mathcal{S}(m, p)$, $p \geq 2$ and it is easy to prove that
$$\min_{(b_1, ..., b_r) \in \mathcal{S}(m, p)} F_v(b_1, ..., b_r; q) = F_v(2_{m - p}, p; q) \mbox{ \cite{BN15a}}.$$
We see that the lower bounding of the vertex Folkman numbers can be achieved by computing or lower bounding the numbers $F_v(2_{m - p}, p; q)$. In general, this is a hard problem. However, in the case $q = m - 1$, in \cite{BN15a} we presented a method for the computation of these numbers, which is based on the following:

\begin{theorem}
\cite{BN15a}
\label{theorem: min_(r geq 2) set(F_v(2_r, p; r + p - 1) - r) = F_v(2_(r_0), p; r_0 + p - 1) - r_0}
Let $r_0(p) = r_0$ be the smallest positive integer for which
\begin{equation*}
\min_{r \geq 2} \set{F_v(2_r, p; r + p - 1) - r} = F_v(2_{r_0}, p; r_0 + p - 1) - r_0.
\end{equation*}
Then:

(a) $F_v(2_r, p; r + p - 1) = F(2_{r_0}, p; r_0 + p - 1) + r - r_0, \ r \geq r_0$.\\

(b) If $r_0 = 2$, then $F_v(2_r, p; r + p - 1) = F_v(2, 2, p; p + 1) + r - 2, \ r \geq 2$.\\

(c) If $r_0 > 2$ and $G$ is an extremal graph in $\mathcal{H} (2_{r_0}, p; r_0 + p - 1)$,\\ then $G \arrowsv (2, r_0 + p - 2)$.\\

(d) $r_0 < F_v(2, 2, p; p + 1) - 2p$.\\
\end{theorem}

From this theorem it becomes clear, that for fixed $p$ the computation of the members of the infinite sequence $F_v(2_{m - p}, p; m - 1)$, $m \geq p + 2$, is reduced to the computation of its first $r_0$ members, where $r_0 < F_v(2, 2, p; p + 1) - 2p$. We conjecture that it is enough to know only its first member $F_v(2, 2, p; p + 1)$.

\begin{conjecture}
\cite{BN15a}
\label{conjecture: F_v(2_r, p; r + p - 1) = F_v(2, 2, p; p + 1) + r - 2}
If $p \geq 4$, then
$$\min_{r \geq 2} \set{F_v(2_r, p; r + p - 1) - r} = F_v(2, 2, p; p + 1) - 2,$$
i.e. $r_0(p) = 2$ and therefore
$$F_v(2_r, p; r + p - 1) = F_v(2, 2, p; p + 1) + r - 2, \quad r \geq 2.$$
\end{conjecture}

This conjecture is proved for $p = 4, 5 \mbox{ and } 6$ in \cite{Nen02b}, \cite{BN15a} and \cite{BN17} respectively. In \cite{BN17} it is also proved that the conjecture is true when $F_v(2, 2, p; p + 1) \leq 2p + 5$. In this paper we will prove, that Conjecture \ref{conjecture: F_v(2_r, p; r + p - 1) = F_v(2, 2, p; p + 1) + r - 2} is also true when $p = 7$:

\begin{theorem}
\label{theorem: F_v(2_(m - 7), 7; m - 1) = m + 11}
$F_v(2_{m - 7}, 7; m - 1) = m + 11$.
\end{theorem}

The Main Theorem follows easily from Theorem \ref{theorem: F_v(2_(m - 7), 7; m - 1) = m + 11}.

\begin{remark}
\label{remark: method for upper bound}
This method is not suitable for upper bounding the vertex Folkman numbers, since it is not clear how $\max_{(b_1, ..., b_r) \in \mathcal{S}(m, p)} F_v(b_1, ..., b_r; q)$ is computed or bounded. In \cite{BN15b} we present another method for upper bounding the vertex Folkman numbers (see also \cite{BN15a} and \cite{BN17}).
\end{remark}

\section{Algorithms}

Finding all graphs in $\mH(a_1, ..., a_s; q; n)$ using a brute force approach is practically impossible for $n > 13$. In this section we present algorithms for obtaining these graphs.

We say that $G$ is a maximal graph in $\mH(a_1, ..., a_s; q)$ if $G \in \mH(a_1, ..., a_s; q)$ but $G + e \not\in \mH(a_1, ..., a_s; q), \forall e \in \E(\overline{G})$, i.e. $\omega(G + e) = q, \forall e \in \E(\overline{G})$. The graphs in $\mH(a_1, ..., a_s; q)$ can be obtained by removing edges from the maximal graphs in this set.

For convenience, we also define the following term:
\begin{definition}
	\label{definition: (+K_t)}
	The graph $G$ is called a $(+K_t)$-graph if $G + e$ contains a new $t$-clique for all $e \in \E(\overline{G})$.
\end{definition}
Obviously, $G \in \mH(a_1, ..., a_s; q)$ is a maximal graph in $\mH(a_1, ..., a_s; q)$ if and only if $G$ is a $(+K_q)$-graph. We shall denote by $\mH_{+K_t}(a_1, ..., a_s; q)$ the set of all $(+K_t)$-graphs in $\mH(a_1, ..., a_s; q)$, and by $\mH_{max}(a_1, ..., a_s; q)$ all maximal $K_q$-free graphs in this set. The sets $\mH_{max}(a_1, ..., a_s; q; n)$ and $\mH_{+K_t}(a_1, ..., a_s; q; n)$ are defined in the same way as $\mH(a_1, ..., a_s; q; n)$.

We shall denote by $\mH_{max}^t(a_1, ..., a_s; q; n)$ and $\mH_{+K_t}^t(a_1, ..., a_s; q; n)$ the subsets of all graphs with independence number not greater than t in the sets $\mH_{max}(a_1, ..., a_s; q; n)$ and $\mH_{+K_t}(a_1, ..., a_s; q; n)$ respectively.

\begin{remark}
\label{remark: mathcal(H)(a_1; q; n) = ...}
In the special case $s = 1$ we have
	
$\mH(a_1; q; n) = \set{ G : a_1 \leq \omega(G) < q \mbox{ and } \abs{\V(G)} = n }$.
	
Obviously, if $a_1 \leq n \leq q - 1$ then $\mH_{max}(a_1; q; n) = \set{K_n}$.

If $a_1 \leq q - 1 \leq n$, then $\mH_{max}(a_1; q; n) = \mH_{max}(q - 1; q; n).$
\end{remark}

Further, we will use the following propositions, which are easy to prove:
\begin{proposition}
\label{proposition: G - A arrowsv (a_1, ..., a_(i - 1), a_i - 1, a_(i + 1_, ..., a_s)}
\cite{BN17}
Let $G$ be a graph, $G \arrowsv (a_1, ..., a_s)$ and $a_i \geq 2$. Then for every independent set $A$ in $G$
\begin{equation*}
G - A \arrowsv (a_1, ..., a_{i - 1}, a_i - 1, a_{i + 1}, ..., a_s).
\end{equation*}
\end{proposition}

\begin{proposition}
\label{proposition: G - A in mH_(+K_(q-1))(a_1 - 1, ..., a_s; q; n - abs(A))}
\cite{BN17}
Let $G \in \mH_{max}(a_1, ..., a_s; q; n)$ and $A$ be an independent set of vertices of $G$. Then $G - A \in \mH_{+K_{q-1}}(a_1 - 1, ..., a_s; q; n - \abs{A})$.
\end{proposition}

In \cite{BN17} we present the following algorithm for finding all graphs $G \in \mH_{max}(a_1, ..., a_s; q; n)$ with $r \leq \alpha(G) \leq t$: 

\begin{algorithm}
	\cite{BN17}
	\label{algorithm: mH_(max)(a_1, ..., a_s; q; n)} The input of the algorithm is the set $\mathcal{A} = \mH_{max}^t(a_1 - 1, ..., a_s; q; n - r)$. The output of the algorithm is the set $\mathcal{B}$ of all graphs $G \in \mH_{max}^t(a_1, ..., a_s; q; n)$ with $\alpha(G) \geq r$.
	
	\emph{1.} By removing edges from the graphs in $\mathcal{A}$ obtain the set
	
	$\mathcal{A}' = \mH_{+K_{q - 1}}^t(a_1 - 1, ..., a_s; q; n - r)$.
	
	\emph{2.} For each graph $H \in \mathcal{A}'$:
	
	\emph{2.1.} Find the family $\mathcal{M}(H) = \set{M_1, ..., M_l}$ of all maximal $K_{q - 1}$-free subsets of $\V(H)$.
	
	\emph{2.2.} Find all $r$-element multisets $N = \set{M_{i_1}, M_{i_2}, ..., M_{i_r}}$ of elements of $\mathcal{M}(H)$, which fulfill the conditions:
	
	(a) $K_{q - 2} \subseteq M_{i_j} \cap M_{i_k}$ for every $M_{i_j}, M_{i_k} \in N$.
	
	(b) $\alpha(H - \bigcup_{M_{i_j} \in N'} M_{i_j}) \leq t - \abs{N'}$ for subtuple $N'$ of $N$.
	
	\emph{2.3.} For each $r$-element multiset $N = \set{M_{i_1}, M_{i_2}, ..., M_{i_r}}$ of elements of $\mathcal{M}(H)$ found in step 2.2 construct the graph $G = G(N)$ by adding new independent vertices $v_1, v_2, ..., v_r$ to $\V(H)$ such that $N_G(v_j) = M_{i_j}, j = 1, ..., r$. If $\omega(G + e) = q, \forall e \in \E(\overline{G})$, then add $G$ to $\mathcal{B}$.
	
	\emph{3.} Remove the isomorph copies of graphs from $\mathcal{B}$.

	\emph{4.} Remove from the obtained in step 3 set $\mathcal{B}$ all graphs $G$ for which $G \not\arrowsv (a_1, ..., a_s)$.
	
\end{algorithm}

\begin{theorem}
	\label{theorem: algorithm mH_(max)(a_1, ..., a_s; q; n)}
	\cite{BN17}
	After the execution of Algorithm \ref{algorithm: mH_(max)(a_1, ..., a_s; q; n)}, the obtained set $\mathcal{B}$ coincides with the set of all graphs $G \in \mH_{max}^t(a_1, ..., a_s; q; n)$ with $\alpha(G) \geq r$.
\end{theorem}

Algorithm \ref{algorithm: mH_(max)(a_1, ..., a_s; q; n)} is based on a very similar algorithm that we used in \cite{BN16} to prove the lower bound $F_e(3, 3; 4) > 19$. It is possible to prove the Main Theorem using Algorithm \ref{algorithm: mH_(max)(a_1, ..., a_s; q; n)}, but it would take us months of computational time. For this reason, we will present an algorithm which is a modification of Algorithm \ref{algorithm: mH_(max)(a_1, ..., a_s; q; n)} and helped us prove the Main Theorem in less than a week on a personal computer.

Further we shall use the following term:
\begin{definition}
\label{definition: cone vertex}
We say that $v$ is a cone vertex in the graph $G$ if $v$ is adjacent to all other vertices in $G$.
\end{definition}

Suppose that $G \in \mH_{max}(a_1, ..., a_s; q; n)$ and $G$ has a cone vertex, i.e. $G = K_1 + H$. According to Proposition \ref{proposition: G - A arrowsv (a_1, ..., a_(i - 1), a_i - 1, a_(i + 1_, ..., a_s)}, $H \in \mH_{max}(a_1 - 1, ..., a_s; q - 1; n - 1)$. Therefore, if we know all the graphs in $\mH_{max}(a_1 - 1, ..., a_s; q - 1; n - 1)$ we can easily obtain the graphs in $\mH_{max}(a_1, ..., a_s; q; n)$ which have a cone vertex. We will use this fact to modify Algorithm \ref{algorithm: mH_(max)(a_1, ..., a_s; q; n)} and make it faster in the case where all graphs in $\mH_{max}(a_1 - 1, ..., a_s; q - 1; n - 1)$ are already known. The new modified algorithm is based on the following:

\begin{proposition}
\label{proposition: G in mH_(max)(a_1, ..., a_s; q; n) and alpha(G) > r}
Let $G \in \mH_{max}(a_1, ..., a_s; q; n)$ be a graph without cone vertices and $A$ be an independent set in $G$ such that $G - A$ has a cone vertex, i.e. $G - A$ = $K_1 + H$. Then $G = \overline{K}_{r + 1} + H$, where $r = \abs{A}$, $H$ has no cone vertices and $K_1 + H \in \mH_{max}(a_1, ..., a_s; q; n - r)$.
\end{proposition}

\begin{proof}
Let $A = \set{v_1, ..., v_r}$ be an independent set in $G$ and $G - A$ = $K_1 + H = \set{u} + H$. Since $G$ has no cone vertices, there exist $v_i \in A$ such that $v_i$ is not adjacent to $u$. Then $N_G(v_i) \subseteq N_G(u)$ and since $G$ is a maximal $K_q$-free graph we obtain $N_G(v_i) = N_G(u) = \V(H)$. Hence, $u$ is not adjacent to any of the vertices in $A$, and therefore $N_G(v_j) = N_G(u) = \V(H), \forall v_j \in A$. We derived $G = \overline{K}_{r + 1} + H$. The graph $H$ has no cone vertices, since any cone vertex in $H$ would be a cone vertex in $G$. It is easy to see that if $\overline{K}_{r + 1} + H \arrowsv (a_1, ..., a_s)$, then $K_1 + H \arrowsv (a_1, ..., a_s)$. Therefore $K_1 + H \in \mH_{max}(a_1, ..., a_s; q; n - r)$.
\end{proof}

Now we present the main algorithm used in this paper, which is a modification of Algorithm \ref{algorithm: mH_(max)(a_1, ..., a_s; q; n)}.

\begin{algorithm}
	\label{algorithm: mH_(max)(a_1, ..., a_s; q; n), cone vertices}
	The input of the algorithm are the set $\mathcal{A}_1 = \mH_{max}^t(a_1 - 1, ..., a_s; q; n - r)$ and the set $\mathcal{A}_2 = \mH_{max}^t(a_1 - 1, ..., a_s; q - 1; n - 1)$. The output of the algorithm is the set $\mathcal{B}$ of all graphs $G \in \mH_{max}^t(a_1, ..., a_s; q; n)$ with $\alpha(G) \geq r$.
	
	\emph{1.} By removing edges from the graphs in $\mathcal{A}_1$ obtain the set
	
	$\mathcal{A}_1' = \set{H \in \mH_{+K_{q - 1}}^t(a_1 - 1, ..., a_s; q; n - r) : \mbox{ $H$ has no cone vertices}}$.
	
	\emph{2.} Repeat step 2 of Algorithm \ref{algorithm: mH_(max)(a_1, ..., a_s; q; n)}.

	\emph{3.} Repeat step 3 of Algorithm \ref{algorithm: mH_(max)(a_1, ..., a_s; q; n)}.

	\emph{4.} Repeat step 4 of Algorithm \ref{algorithm: mH_(max)(a_1, ..., a_s; q; n)}.
	
	\emph{5.} If $t > r$, find the subset $\mathcal{A}_1''$ of $\mathcal{A}_1$ containing all graphs with exactly one cone vertex. For each graph $H \in \mathcal{A}_1''$, if $K_1 + H \arrowsv (a_1, ..., a_s)$, then add $\overline{K}_{r + 1} + H$ to $\mathcal{B}$.

	\emph{6.} For each graph $H$ in $\mathcal{A}_2$ such that $\alpha(H) \geq r$, if $K_1 + H \arrowsv (a_1, ..., a_s)$, then add $K_1 + H$ to $\mathcal{B}$.
	
\end{algorithm}

\begin{theorem}
\label{theorem: algorithm mH_(max)(a_1, ..., a_s; q; n), cone vertices}
After the execution of Algorithm \ref{algorithm: mH_(max)(a_1, ..., a_s; q; n), cone vertices}, the obtained set $\mathcal{B}$ coincides with the set of all graphs $G \in \mH_{max}^t(a_1, ..., a_s; q; n)$ with $\alpha(G) \geq r$.
\end{theorem}

\begin{proof}
Suppose that after the execution of Algorithm \ref{algorithm: mH_(max)(a_1, ..., a_s; q; n), cone vertices}, $G \in \mathcal{B}$. If after step 4 $G \in \mathcal{B}$, then according to Theorem \ref{theorem: algorithm mH_(max)(a_1, ..., a_s; q; n)}, $G \in \mH_{max}^t(a_1, ..., a_s; q; n)$ and $\alpha(G) \geq r$. If $G$ is added to $\mathcal{B}$ in step 5 or step 6, then clearly $G \in \mH_{max}^t(a_1, ..., a_s; q; n)$ and $\alpha(G) \geq r$.

Now let $G \in \mH_{max}^t(a_1, ..., a_s; q; n)$ and $\alpha(G) \geq r$. If $G = K_1 + H$ for some graph $H$, then, according to Proposition \ref{proposition: G - A arrowsv (a_1, ..., a_(i - 1), a_i - 1, a_(i + 1_, ..., a_s)}, $H \in \mathcal{A}_2$ and in step 6 $G$ is added to $\mathcal{B}$. Suppose that $G$ has no cone vertices and $G$ has an independent set $A$ such that $\abs{A} = r$ and $G - A$ has a cone vertex, i.e. $G - A = K_1 + H$. Then, according to Proposition \ref{proposition: G in mH_(max)(a_1, ..., a_s; q; n) and alpha(G) > r}, $G = \overline{K}_{r + 1} + H$, $K_1 + H$ has exactly one cone vertex and $K_1 + H \arrowsv (a_1, ..., a_s)$. It is clear that $t > r$ and hence in step 5 $G$ is added to $\mathcal{B}$. Finally, if $G - A$ has no cone vertices, then according to Proposition \ref{proposition: G - A in mH_(+K_(q-1))(a_1 - 1, ..., a_s; q; n - abs(A))}, $G - A \in \mathcal{A}_1'$ and it follows from Theorem \ref{theorem: algorithm mH_(max)(a_1, ..., a_s; q; n)} that after the execution of step 4, $G \in \mathcal{B}$.
\end{proof}

\begin{remark}
\label{remark: algorithms case r = 2}
Note that if $n \geq q$ and $r = 2$ Algorithm \ref{algorithm: mH_(max)(a_1, ..., a_s; q; n)} and Algorithm \ref{algorithm: mH_(max)(a_1, ..., a_s; q; n), cone vertices} obtain all graphs in $G \in \mH_{max}^t(a_1, ..., a_s; q; n)$.
\end{remark}

The \emph{nauty} programs \cite{MP13} have an important role in this paper. We use them for fast generation of non-isomorph graphs and isomorph rejection.

\section{Proof of the Main Theorem and Theorem 2.3}

We will first prove Theorem \ref{theorem: F_v(2_(m - 7), 7; m - 1) = m + 11} by proving Conjecture \ref{conjecture: F_v(2_r, p; r + p - 1) = F_v(2, 2, p; p + 1) + r - 2} in the case $p = 7$. Since $F_v(2, 2, 7; 8) = 20$ \cite{BN17}, according to Theorem \ref{theorem: min_(r geq 2) set(F_v(2_r, p; r + p - 1) - r) = F_v(2_(r_0), p; r_0 + p - 1) - r_0}(d), to prove the conjecture in this case we need to prove the inequalities $F_v(2, 2, 2, 7; 9) > 20$, $F_v(2, 2, 2, 2, 7; 10) > 21$ and $F_v(2, 2, 2, 2, 2, 7; 11) > 22$. It is easy to see that it is enough to prove only the last of the three inequalities (see \cite{BN17} for details). Using Algorithm \ref{algorithm: mH_(max)(a_1, ..., a_s; q; n)} it can be proved that $F_v(2, 2, 2, 2, 2, 7; 11) > 22$ , but it would require a lot of computational time. Instead, we will prove the three inequalities successively using Algorithm \ref{algorithm: mH_(max)(a_1, ..., a_s; q; n), cone vertices}. Only the proof of the first inequality is presented in details, since the proofs of the others are very similar. We will show that $\mH(2, 2, 7; 8; 19) = \emptyset$. The proof uses the graphs $\mH_{max}^3(4; 8; 8)$, $\mH_{max}^3(5; 8; 10)$, $\mH_{max}^3(6; 8; 12)$, $\mH_{max}^3(7; 8; 14)$, $\mH_{max}^3(2, 7; 8; 16)$, $\mH_{max}^3(2, 2, 7; 8; 19)$, $\mH_{max}^2(4; 8; 9)$, $\mH_{max}^2(5; 8; 11)$, $\mH_{max}^2(6; 8; 13)$, $\mH_{max}^2(7; 8; 15)$, $\mH_{max}^2(2, 7; 8; 17)$, $\mH_{max}^2(2, 2, 7; 8; 19)$ obtained in \cite{BN17} in the proof of the lower bound $F_v(2, 2, 7; 8) \geq 20$ (see Table \ref{table: finding all graphs in H(2, 2, 7; 8; 19)}).

For positive integers $a_1, ..., a_s$ and $m$ and $p$ defined by (\ref{equation: m and p}), Nenov proved in \cite{Nen02a} that if $G \in \mH(a_1, ..., a_s; m - 1; n)$ and $n < m + 3p$, then $\alpha(G) < n - m - p + 1$. 
Suppose that $G \in \mH(2, 2, 2, 7; 9; 20)$. It follows that $\alpha(G) \leq 3$ and it is clear that $\alpha(G) \geq 2$. Therefore, it is enough to prove that there are no graphs with independence number 2 or 3 in $\mH_{max}(2, 2, 2, 7; 9; 20)$.

First we prove that there are no graphs in $\mH_{max}(2, 2, 2, 7; 9; 20)$ with independence number 3. It is clear that $K_7$ is the only graph in $\mH_{max}(4; 9; 7)$. By applying Algorithm \ref{algorithm: mH_(max)(a_1, ..., a_s; q; n), cone vertices}($r = 2; t = 3$) with $\mathcal{A}_1 = \mH_{max}^3(4; 9; 7) = \set{K_7}$ and $\mathcal{A}_2 = \mH_{max}^3(4; 8; 8)$ were obtained all graphs in $\mH_{max}^3(5; 9; 9)$ (see Remark \ref{remark: algorithms case r = 2}). In the same way, we successively obtained all graphs in $\mH_{max}^3(6; 9; 11)$, $\mH_{max}^3(7; 9; 13)$, $\mH_{max}^3(2, 7; 9; 15)$ and $\mH_{max}^3(2, 2, 7; 9; 17)$ (see Remark \ref{remark: algorithms case r = 2}). In the end, by applying Algorithm \ref{algorithm: mH_(max)(a_1, ..., a_s; q; n), cone vertices}($r = 3; t = 3$) with $\mathcal{A}_1 = \mH_{max}^3(2, 2, 7; 9; 17)$ and $\mathcal{A}_2 = \mH_{max}^3(2, 2, 7; 8; 19) = \emptyset$ no graphs with independence number 3 in $\mH_{max}(2, 2, 2, 7; 9; 20)$ were obtained.

It remains to prove that there are no graphs in $\mH_{max}(2, 2, 2, 7; 9; 20)$ with independence number 2. Clearly, $K_8$ is the only graph in $\mH_{max}(4; 9; 8)$. By applying Algorithm \ref{algorithm: mH_(max)(a_1, ..., a_s; q; n), cone vertices}($r = 2; t = 2$) with $\mathcal{A}_1 = \mH_{max}^2(4; 9; 8) = \set{K_8}$ and $\mathcal{A}_2 = \mH_{max}^2(4; 8; 9)$ were obtained all graphs in $\mH_{max}^2(5; 9; 10)$ (see Remark \ref{remark: algorithms case r = 2}). In the same way, we successively obtained all graphs in $\mH_{max}^2(6; 9; 12)$, $\mH_{max}^2(7; 9; 14)$, $\mH_{max}^2(2, 7; 9; 16)$ and $\mH_{max}^2(2, 2, 7; 9; 18)$ (see Remark \ref{remark: algorithms case r = 2}). In the end, by applying Algorithm \ref{algorithm: mH_(max)(a_1, ..., a_s; q; n), cone vertices}($r = 2; t = 2$) with $\mathcal{A}_1 = \mH_{max}^2(2, 2, 7; 9; 18)$ and $\mathcal{A}_2 = \mH_{max}^2(2, 2, 7; 8; 19) = \emptyset$ no graphs with independence number 2 in $\mH_{max}(2, 2, 2, 7; 9; 20)$ were obtained.

We proved that $\mH_{max}(2, 2, 2, 7; 9; 20) = \emptyset$ and $F_v(2, 2, 2, 7; 9) > 20$.

In the same way, the graphs obtained in the proof of the inequality $F_v(2, 2, 2, 7; 9) > 20$ are used to prove $F_v(2, 2, 2, 2, 7; 10) > 21$ and the graphs obtained in the proof of the inequality $F_v(2, 2, 2, 2, 7; 10) > 21$ are used to prove $F_v(2, 2, 2, 2, 2, 7; 11) > 22$. The number of graphs obtained in each step of the proofs is shown in Table \ref{table: finding all graphs in H(2, 2, 2, 7; 9; 20)}, Table \ref{table: finding all graphs in H(2, 2, 2, 2, 7; 10; 21)} and Table \ref{table: finding all graphs in H(2, 2, 2, 2, 2, 7; 11; 22)}. Notice that the number of graphs without cone vertices is relatively small, which reduces the computation time significantly.\\

Thus, $r_0(7) = 2$ and
$$F_v(2_{m - 7}, 7; m - 1) = F_v(2, 2, 7; 8) + m - 9 = m + 11$$
which finishes the proof of Theorem \ref{theorem: F_v(2_(m - 7), 7; m - 1) = m + 11}. The Main Theorem follows easily. Indeed, let $a_1, ..., a_s$ be positive integers such that $\max\set{a_1, ..., a_s} = 7$ and $m = \sum\limits_{i=1}^s (a_i - 1) + 1$. Then
$$F_v(a_1, ..., a_s; m - 1) \geq F_v(2_{m - 7}, 7; m - 1) =  m + 11.$$
\vspace{-10pt}\qed

\section{Concluding remarks}

The considered method for lower bounding the numbers $F_v(a_1, ..., a_s; q)$ gives good and accurate results when $q = m - 1$. However, when $q < m - 1$ the bounds are not exact. We will consider the most interesting case $q = p + 1$, where $p = \max\set{a_1, ..., a_s}$. In \cite{BN15a} we prove the inequality
\begin{equation}
\label{equation: F_v(a_1, ..., a_s; p + 1) geq F_v(2, 2, p; p + 1) + sum_(i = 3)^(m - p) alpha(i, p)}
F_v(a_1, ..., a_s; p + 1) \geq F_v(2, 2, p; p + 1) + \sum_{i = 3}^{m - p}\alpha(i, p),
\end{equation}
where $\alpha(i, p) = \max\set{\alpha(G) : G \in \mH_{extr}(2_i, p; p + 1)}$. Since $\alpha(i, p) \geq 2$, from (\ref{equation: F_v(a_1, ..., a_s; p + 1) geq F_v(2, 2, p; p + 1) + sum_(i = 3)^(m - p) alpha(i, p)}) it follows
$$F_v(a_1, ..., a_s; p + 1) \geq F_v(2, 2, p; p + 1) + 2(m - p - 2).$$
In the special case $p = 7$, since $F_v(2, 2, 7; 8) = 20$ we obtain
\begin{equation}
\label{equation: F_v(a_1, ..., a_s; 8) geq 2m + 2}
F_v(a_1, ..., a_s; 8) \geq 2m + 2.
\end{equation}
In particular, when $m = 13$ we have $F_v(a_1, ..., a_s; 8) \geq 28$. Since the Ramsey number $R(3, 8) = 28$, it follows that $\alpha(i, 7) \geq 3$, when $i \geq 6$. Now from (\ref{equation: F_v(a_1, ..., a_s; p + 1) geq F_v(2, 2, p; p + 1) + sum_(i = 3)^(m - p) alpha(i, p)}) it follows easily that
\begin{theorem}
\label{theorem: F_v(a_1, ..., a_s; 8) geq 3m - 10}
If $m \geq 13$, and $\max\set{a_1, ..., a_s} = 7$, then
$$F_v(a_1, ..., a_s; 8) \geq 3m - 10.$$
\end{theorem}
It is clear that when $3m - 10 \geq R(4, 8)$ these bounds for $F_v(a_1, ..., a_s; 8)$ can be improved considerably.

In \cite{XS10} it is proved the inequality $F_v(p, p; p + 1) \geq 4p - 1$. From this result it follows that $F_v(7, 7; 8) \geq 27$. From (\ref{equation: F_v(a_1, ..., a_s; 8) geq 2m + 2}) we obtain $F_v(7, 7; 8) \geq 28$, and from Theorem \ref{theorem: F_v(a_1, ..., a_s; 8) geq 3m - 10} we obtain $F_v(7, 7; 8) \geq 29$.

The numbers $F_v(p, p; p + 1)$ are of significant interest, but so far we know very little about them. Only two of these numbers are known, $F_v(2, 2; 3) = 5$ (obvious), and $F_v(3, 3; 4) = 14$ (\cite{Nen81} and \cite{PRU99}). It is also known that $17 \leq F_v(4, 4; 5) \leq 23$, \cite{XLS10}, $F_v(5, 5; 6) \geq 23$, \cite{BN15a}, $28 \leq F_v(6, 6; 7) \leq 70$, \cite{BN17}, and $F_v(7, 7; 8) \geq 29$ from this paper. Using Algorithm \ref{algorithm: mH_(max)(a_1, ..., a_s; q; n)}, we managed to improve the known lower bound $F_v(2, 2, 2, 4; 5) \geq 17$ and thus improved the lower bound on $F_v(4, 4; 5)$ as well:

\begin{theorem}
\label{theorem: F_v(4, 4; 5) geq F_v(2, 3, 4; 5) geq F_v(2, 2, 2, 4; 5) geq 19}
$F_v(4, 4; 5) \geq F_v(2, 3, 4; 5) \geq F_v(2, 2, 2, 4; 5) \geq 19$.
\end{theorem}

\begin{proof}
The inequalities $F_v(4, 4; 5) \geq F_v(2, 3, 4; 5) \geq F_v(2, 2, 2, 4; 5)$ are easy to prove (see (4.1) in \cite{BN15a}). It remains to prove that $F_v(2, 2, 2, 4; 5) \geq 19$. Suppose that $\mH_{max}(2, 2, 2, 4; 5; 18) \neq \emptyset$ and let $G \in \mH_{max}(2, 2, 2, 4; 5; 18)$. Since the Ramsey number $R(3, 5) = 14$, $\alpha(G) \geq 3$. In \cite{XLS10} it is proved that $F_v(2, 2, 4; 5) = 13$ and $\mH(2, 2, 4; 5; 13) = \set{Q}$, where $Q$ is the unique 13-vertex $K_5$-free graph with independence number 2. From Proposition \ref{proposition: G - A arrowsv (a_1, ..., a_(i - 1), a_i - 1, a_(i + 1_, ..., a_s)} and the equality $F_v(2, 2, 4; 5) = 13$ it follows that $\alpha(G) \leq 5$. By applying Algorithm \ref{algorithm: mH_(max)(a_1, ..., a_s; q; n)} to the graph $Q$ it follows that there are no graphs in $\mH_{max}(2, 2, 2, 4; 5; 18)$ with independence number 5. It remains to prove that there are no graphs in $\mH_{max}(2, 2, 2, 4; 5; 18)$ with independence number 3 or 4. Using \emph{nauty} it is easy to obtain the sets $\mH_{max}^4(3; 5; 8)$ and $\mH_{max}^3(3; 5; 9)$. By applying Algorithm \ref{algorithm: mH_(max)(a_1, ..., a_s; q; n)} ($r = 2, t = 4$) starting from the set $\mH_{max}^4(3; 5; 8)$ were successively obtained all graphs in the sets $\mH_{max}^4(4; 5; 10)$, $\mH_{max}^4(2, 4; 5; 12)$, $\mH_{max}^4(2, 2, 4; 5; 14)$  (see Remark \ref{remark: algorithms case r = 2}), and by applying Algorithm \ref{algorithm: mH_(max)(a_1, ..., a_s; q; n)} ($r = 4, t = 4$) were found no graphs in $\mH_{max}(2, 2, 2, 4; 5; 18)$ with independence number 4. Next, we applied Algorithm \ref{algorithm: mH_(max)(a_1, ..., a_s; q; n)} ($r = 2, t = 3$) starting from the set $\mH_{max}^3(3; 5; 9)$ to successively obtain all graphs in the sets $\mH_{max}^3(4; 5; 11)$, $\mH_{max}^3(2, 4; 5; 13)$, $\mH_{max}^3(2, 2, 4; 5; 15)$ (see Remark \ref{remark: algorithms case r = 2}), and by applying Algorithm \ref{algorithm: mH_(max)(a_1, ..., a_s; q; n)} ($r = 3, t = 3$) were found no graphs in $\mH_{max}(2, 2, 2, 4; 5; 18)$ with independence number 3. The number of graphs obtained in each of the steps is shown in Table \ref{table: finding all graphs in H(2, 2, 2, 4; 5; 18)}. We obtained $\mH_{max}(2, 2, 2, 4; 5; 18) = \emptyset$ and therefore $F_v(2, 2, 2, 4; 5) \geq 19$.
\end{proof}

The upper bound $F_v(4, 4; 5) \leq 23$ is proved in \cite{XLS10} with the help of a 23-vertex transitive graph. We were not able to obtain any other graphs in $\mH(4, 4; 5; 23)$, which leads us to believe that this bound may be exact. We did find a large number of 23-vertex graphs in $\mH(2, 2, 2, 4; 5)$, but so far we have not obtained smaller graphs in this set.

In the end, we shall pose the following question:

\emph{Is it true, that the sequence $F_v(p, p; p + 1), p \geq 2,$ is increasing?}

% -----------------------------------------------------------

\bigskip\medskip

ACKNOWLEDGEMENT. This paper was partially supported by the Sofia University Scientific Research Fund through Contract 80-10-74/20.04.2017.

\vspace{2em}
\begin{appendices}
	
	\section{Results of computations}
	
	\vspace{2em}
	
	\begin{table}[h]
		\centering
		\resizebox{0.75\textwidth}{!}{
		\begin{tabular}{| p{3.5cm} | p{1.5cm} | p{2cm} | p{2.5cm} |}
			\hline
			set							& ind. number			& maximal graphs 		& $(+K_7)$-graphs	\\
			\hline
			$\mH(2, 7; 8; 15)$			& $\leq 4$				& 1						& 1					\\
			$\mH(2, 2, 7; 8; 19)$		& $= 4$					& 0						&					\\
			\hline
			$\mH(3; 8; 6)$				& $\leq 3$				& 1						& 1					\\
			$\mH(4; 8; 8)$				& $\leq 3$				& 1						& 4					\\
			$\mH(5; 8; 10)$				& $\leq 3$				& 3						& 45				\\
			$\mH(6; 8; 12)$				& $\leq 3$				& 12					& 3 104				\\
			$\mH(7; 8; 14)$				& $\leq 3$				& 169					& 4 776 518			\\
			$\mH(2, 7; 8; 16)$			& $\leq 3$				& 34					& 22 896			\\
			$\mH(2, 2, 7; 8; 19)$		& $= 3$					& 0						&					\\
			\hline
			$\mH(3; 8; 7)$				& $\leq 2$				& 1						& 1					\\
			$\mH(4; 8; 9)$				& $\leq 2$				& 1						& 8					\\
			$\mH(5; 8; 11)$				& $\leq 2$				& 3						& 84				\\
			$\mH(6; 8; 13)$				& $\leq 2$				& 10					& 5 394				\\
			$\mH(7; 8; 15)$				& $\leq 2$				& 102					& 4 984 994			\\
			$\mH(2, 7; 8; 17)$			& $\leq 2$				& 2760					& 380 361 736		\\
			$\mH(2, 2, 7; 8; 19)$		& $= 2$					& 0						&					\\
			\hline
			$\mH(2, 2, 7; 8; 19)$		& 						& 0						&					\\
			\hline
		\end{tabular}
		}
		\caption{Steps in finding all maximal graphs in $\mH(2, 2, 7; 8; 19)$}
		\label{table: finding all graphs in H(2, 2, 7; 8; 19)}
	\end{table}
	
	\vspace{2em}
	
	\begin{table}[h]
		\centering
		\resizebox{0.95\textwidth}{!}{
		\begin{tabular}{| p{3.5cm} | p{1.5cm} | p{2cm} | p{2cm} | p{2.5cm} | p{2.5cm} |}
			\hline
			set	& ind. number			& max. graphs	& max. graphs no cone v.	& $(+K_8)$-graphs	& $(+K_8)$-graphs no cone v.	\\
			\hline
			$\mH(2, 2, 7; 9; 16)$		& $\leq 4$				& 1					& 0		& 1					& 0			\\
			$\mH(2, 2, 2, 7; 9; 20)$	& $= 4$					& 0					& 0		&					&			\\
			\hline
			$\mH(4; 9; 7)$				& $\leq 3$				& 1					& 0		& 1					& 0			\\
			$\mH(5; 9; 9)$				& $\leq 3$				& 1					& 0		& 4					& 0			\\
			$\mH(6; 9; 11)$				& $\leq 3$				& 3					& 0		& 45				& 0			\\
			$\mH(7; 9; 13)$				& $\leq 3$				& 12				& 0		& 3 113				& 9			\\
			$\mH(2, 7; 9; 15)$			& $\leq 3$				& 169				& 0		& 4 783 615			& 7 097		\\
			$\mH(2, 2, 7; 9; 17)$		& $\leq 3$				& 36				& 2		& 22 918			& 22		\\
			$\mH(2, 2, 2, 7; 9; 20)$	& $= 3$					& 0					& 0		&					&			\\
			\hline
			$\mH(4; 9; 8)$				& $\leq 2$				& 1					& 0		& 1					& 0			\\
			$\mH(5; 9; 10)$				& $\leq 2$				& 1					& 0		& 8					& 0			\\
			$\mH(6; 9; 12)$				& $\leq 2$				& 3					& 0		& 85				& 1			\\
			$\mH(7; 9; 14)$				& $\leq 2$				& 10				& 0		& 5 474				& 80		\\
			$\mH(2, 7; 9; 16)$			& $\leq 2$				& 103				& 1		& 5 346 982			& 361 988	\\
			$\mH(2, 2, 7; 9; 18)$		& $\leq 2$				& 2845				& 85	& 387 948 338  		& 7 586 602	\\
			$\mH(2, 2, 2, 7; 9; 20)$	& $= 2$					& 0					& 0		&					&			\\
			\hline
			$\mH(2, 2, 2, 7; 9; 20)$	& 						& 0					& 0		&					&			\\
			\hline
		\end{tabular}
		}
		\caption{Steps in finding all maximal graphs in $\mH(2, 2, 2, 7; 9; 20)$}
		\label{table: finding all graphs in H(2, 2, 2, 7; 9; 20)}
	\end{table}
	
	\begin{table}[h]
		\centering
		\resizebox{0.95\textwidth}{!}{
		\begin{tabular}{| p{3.5cm} | p{1.5cm} | p{2cm} | p{2cm} | p{2.5cm} | p{2.5cm} |}
			\hline
			set	& ind. number				& max. graphs	& max. graphs no cone v.	& $(+K_9)$-graphs	& $(+K_9)$-graphs no cone v.	\\
			\hline
			$\mH(2, 2, 2, 7; 10; 17)$		& $\leq 4$				& 1					& 0		& 1					& 0			\\
			$\mH(2, 2, 2, 2, 7; 10; 21)$	& $= 4$					& 0					& 0		&					&			\\
			\hline
			$\mH(5; 10; 8)$					& $\leq 3$				& 1					& 0		& 1					& 0			\\
			$\mH(6; 10; 10)$				& $\leq 3$				& 1					& 0		& 4					& 0			\\
			$\mH(7; 10; 12)$				& $\leq 3$				& 3					& 0		& 45				& 0			\\
			$\mH(2, 7; 10; 14)$				& $\leq 3$				& 12				& 0		& 3 115				& 2			\\
			$\mH(2, 2, 7; 10; 16)$			& $\leq 3$				& 169				& 0		& 4 784 483			& 868		\\
			$\mH(2, 2, 2, 7; 10; 18)$		& $\leq 3$				& 36				& 0		& 22 919			& 1			\\
			$\mH(2, 2, 2, 2, 7; 10; 21)$	& $= 3$					& 0					& 0		&					&			\\
			\hline
			$\mH(5; 10; 9)$					& $\leq 2$				& 1					& 0		& 1					& 0			\\
			$\mH(6; 10; 11)$				& $\leq 2$				& 1					& 0		& 8					& 0			\\
			$\mH(7; 10; 13)$				& $\leq 2$				& 3					& 0		& 85				& 0			\\
			$\mH(2, 7; 10; 15)$				& $\leq 2$				& 10				& 0		& 5 495				& 21		\\
			$\mH(2, 2, 7; 10; 17)$			& $\leq 2$				& 103				& 0		& 5 371 651			& 24 669	\\
			$\mH(2, 2, 2, 7; 10; 19)$		& $\leq 2$				& 2848				& 3		& 387 968 658  		& 20 320	\\
			$\mH(2, 2, 2, 2, 7; 10; 21)$	& $= 2$					& 0					& 0		&					&			\\
			\hline
			$\mH(2, 2, 2, 2, 7; 10; 21)$	& 						& 0					& 0		&					&			\\
			\hline
		\end{tabular}
		}
		\caption{Steps in finding all maximal graphs in $\mH(2, 2, 2, 2, 7; 10; 21)$}
		\label{table: finding all graphs in H(2, 2, 2, 2, 7; 10; 21)}
	\end{table}
	
	\begin{table}[h]
		\centering
		\resizebox{0.95\textwidth}{!}{
		\begin{tabular}{| p{3.5cm} | p{1.5cm} | p{2cm} | p{2cm} | p{2.5cm} | p{2.5cm} |}
			\hline
			set	& ind. number					& max. graphs	& max. graphs no cone v.	& $(+K_{10})$-graphs	& $(+K_{10})$-graphs no cone v.	\\
			\hline
			$\mH(2, 2, 2, 2, 7; 11; 18)$		& $\leq 4$				& 1					& 0		& 1					& 0			\\
			$\mH(2, 2, 2, 2, 2, 7; 11; 22)$		& $= 4$					& 0					& 0		&					&			\\
			\hline
			$\mH(6; 11; 9)$						& $\leq 3$				& 1					& 0		& 1					& 0			\\
			$\mH(7; 11; 11)$					& $\leq 3$				& 1					& 0		& 4					& 0			\\
			$\mH(2, 7; 11; 13)$					& $\leq 3$				& 3					& 0		& 45				& 0			\\
			$\mH(2, 2, 7; 11; 15)$				& $\leq 3$				& 12				& 0		& 3 116				& 1			\\
			$\mH(2, 2, 2, 7; 11; 17)$			& $\leq 3$				& 169				& 0		& 4 784 638			& 155		\\
			$\mH(2, 2, 2, 2, 7; 11; 19)$		& $\leq 3$				& 36				& 0		& 22 919			& 0			\\
			$\mH(2, 2, 2, 2, 2, 7; 11; 22)$		& $= 3$					& 0					& 0		&					&			\\
			\hline
			$\mH(6; 11; 10)$					& $\leq 2$				& 1					& 0		& 1					& 0			\\
			$\mH(7; 11; 12)$					& $\leq 2$				& 1					& 0		& 8					& 0			\\
			$\mH(2, 7; 11; 14)$					& $\leq 2$				& 3					& 0		& 85				& 0			\\
			$\mH(2, 2, 7; 11; 16)$				& $\leq 2$				& 10				& 0		& 5 502				& 7			\\
			$\mH(2, 2, 2, 7; 11; 18)$			& $\leq 2$				& 103				& 0		& 5 374 143			& 2 492		\\
			$\mH(2, 2, 2, 2, 7; 11; 20)$		& $\leq 2$				& 2848				& 0		& 387 968 676  		& 18		\\
			$\mH(2, 2, 2, 2, 2, 7; 11; 22)$		& $= 2$					& 0					& 0		&					&			\\
			\hline
			$\mH(2, 2, 2, 2, 2, 7; 11; 22)$		& 						& 0					& 0		&					&			\\
			\hline
		\end{tabular}
		}
		\caption{Steps in finding all maximal graphs in $\mH(2, 2, 2, 2, 2, 7; 11; 22)$}
		\label{table: finding all graphs in H(2, 2, 2, 2, 2, 7; 11; 22)}
	\end{table}
	
	\begin{table}[h]
		\centering
		\resizebox{0.75\textwidth}{!}{
		\begin{tabular}{| p{3.5cm} | p{1.5cm} | p{2cm} | p{2.5cm} |}
			\hline
			set							& ind. number			& maximal graphs 		& $(+K_4)$-graphs	\\
			\hline
			$\mH(2, 2, 4; 5; 13)$		& $\leq 5$				& 1						& 1					\\
			$\mH(2, 2, 2, 4; 5; 18)$	& $= 5$					& 0						&					\\
			\hline
			$\mH(3; 5; 8)$				& $\leq 4$				& 7						& 274				\\
			$\mH(4; 5; 10)$				& $\leq 4$				& 44					& 65 422			\\
			$\mH(2, 4; 5; 12)$			& $\leq 4$				& 1 059					& 18 143 174		\\
			$\mH(2, 2, 4; 5; 14)$		& $\leq 4$				& 13					& 71				\\
			$\mH(2, 2, 2, 4; 5; 18)$	& $= 4$					& 0						&					\\
			\hline
			$\mH(3; 5; 9)$				& $\leq 3$				& 11					& 2 252				\\
			$\mH(4; 5; 11)$				& $\leq 3$				& 135					& 1 678 802			\\
			$\mH(2, 4; 5; 13)$			& $\leq 3$				& 11 439				& 2 672 047 607		\\
			$\mH(2, 2, 4; 5; 15)$		& $\leq 3$				& 1 103					& 78 117			\\
			$\mH(2, 2, 2, 4; 5; 18)$	& $= 3$					& 0						&					\\
			\hline
			$\mH(2, 2, 2, 4; 5; 18)$	& 						& 0						&					\\
			\hline
		\end{tabular}
		}
		\caption{Steps in finding all maximal graphs in $\mH(2, 2, 2, 4; 5; 18)$}
		\label{table: finding all graphs in H(2, 2, 2, 4; 5; 18)}
	\end{table}
	
\end{appendices}

\clearpage

{\small

  \vspace{12pt}
\baselineskip10pt

\vskip10pt
%\begin{flushright}
%{\it Received on September 16, 2017}
%\end{flushright}
\vskip20pt
 \footnotesize
\begin{flushleft}
Aleksandar Bikov, Nedyalko Nenov\smallskip\\
Faculty of Mathematics and Informatics \\
``St.~Kl.~Ohridski" University of Sofia \\
5,~J.~Bourchier blvd., BG-1164 Sofia\\
BULGARIA \smallskip\\
e-mails: $\begin{array}{l}
\text{asbikov@fmi.uni-sofia.bg}\\
\text{nenov@fmi.uni-sofia.bg}
        \end{array}$
\end{flushleft}}

\end{document}